\documentclass{article}
\usepackage{graphicx}
\usepackage{amsmath, amssymb, amsthm}
\usepackage[colorlinks=true, allcolors=blue]{hyperref}
\usepackage{verbatim}
\usepackage{mathtools, mathrsfs}
\usepackage{enumitem}
\usepackage{cleveref}
\usepackage{tikz}
\usepackage{tkz-euclide}
\usepackage[normalem]{ulem}

\theoremstyle{plain}

\newtheorem{lemma}{Lemma}
\newtheorem{corollary}{Corollary}
\newtheorem{theorem}{Theorem}

\theoremstyle{definition}
\newtheorem{definition}{Definition}

\theoremstyle{remark}

\title{Preserving and Increasing Symmetries of Polyhedral Maps}
\author{Gunnar Brinkmann, Fabio Buccoliero, Heidi Van den Camp}
\date{}

\begin{document}

\maketitle

\section*{Abstract}

In this article we investigate the question which local symmetry preserving operations can not only preserve, but also increase the symmetry of a polyhedral
map.  Often operations that can increase symmetry, can nevertheless not do so for polyhedral maps of every genus. So for maps that can increase symmetry,
we also investigate for which genera they can do so. We give complete answers for operations with inflation factor at most 6 (that is:
that increase the number of edges by a factor of at most 6) and for the chemically relevant Goldberg-Coxeter operations and the leapfrog operation.

\section*{Introduction}

Polyhedral maps and operations on polyhedra became especially relevant in chemistry after the discovery of fullerenes. For fullerenes, Goldberg-Coxeter
operations and the leapfrog operation were used to construct larger fullerenes from smaller ones -- preserving the symmetry group.  In the general framework that we
use, the leapfrog operation is in a fact a special Goldberg-Coxeter operation.  The result of a leapfrog operation applied to any fullerene -- no matter
of which symmetry -- not only has the same symmetry group, but also a closed shell \cite{Fullbook}, which makes this operation especially interesting and well
studied. Next to fullerene polyhedra, also higher genus analogues of Fullerenes have been studied (see e.g. \cite{highgenusfull}) and -- at least for the torus
-- even observed in nature \cite{torus97}. So it is interesting whether these operations also just preserve the symmetry on maps of higher genus. We will
answer this question even in a more general context -- that of {\em local symmetry preserving operations}.

Local symmetry preserving operations such as {\em truncation}, {\em ambo} or {\em dual} were most likely already known to the ancient Greeks, who described the
Platonic solids and the Archimedean solids which can be constructed from the Platonic solids by local operations preserving symmetries of the original
object. When rediscovering the Archimedean solids, also Kepler used (and named) such operations in his book {\em Harmonices Mundi} \cite{Kepler}.
Some operations not only preserve symmetries, but sometimes even increase it.
For classical polyhedra (that
is: $3$-dimensional convex polyhedra) the only known operation where this happens is ambo (or a combination of ambo with another operation) applied to self-dual
polyhedra.  For polyhedral maps on surfaces of higher genus, this can also happen with other operations, as shown in \cite{korbitmaps}. In 2017, a general
description of local symmetry preserving operations encompassing all known operations was given \cite{lspgocox}. That definition made it possible to give
complete lists of local symmetry preserving operations that increase the number of edges by a certain factor -- the {\em inflation factor}. In this article we determine on
which genera operations with a small inflation factor can increase the symmetry of a polyhedral map. Goldberg-Coxeter operations were independently -- and in
slightly different ways -- introduced by Goldberg \cite{goldberg1937} in a mathematical context and by Caspar and Klug \cite{caspar1962} in a biological
context. Later these operations also became relevant for chemistry, to construct all fullerenes with icosahedral symmetry. Goldberg-Coxeter operations are
described by two parameters and there is an infinite number of them. We determine for which parameters and genera they can increase the symmetry of a polyhedral map.

\section{Definitions}

The term \emph{polyhedron} is used in different ways in the literature. As the planar case is often special, we will use the term in the classical way only for
maps corresponding to $3$-dimensional convex polyhedra -- that is due to Steinitz' theorem: $3$-connected simple graphs embedded in the plane.
For the more
general case of a $3$-connected graph $G$ embedded in a $2$-dimensional surface $S$ of possibly higher genus such that the closure of every face (that is: a
component of $S\setminus G$) is a closed disk and the intersection of the closure of two faces of the map is connected, we will use the term {\em
  polyhedral map}.  An equivalent definition of polyhedral map is a 3-connected embedded graph of face-width -- also known as representativity -- at least 3
\cite{mohar1997face}.  The \emph{boundary} of a face $f$, denoted by $\partial f$, is the closed walk of a polyhedral map $P$ which is obtained intersecting $P$ and the
topological closure of $f$.  A \emph{rotation system} is the specification of a circular ordering of the edges incident at each vertex of a map.  There
is a one-to-one correspondence between homeomorphism classes of maps on oriented surfaces and rotation systems \cite{top_graph_theory}\cite{graphs_on_surfaces}.

The \emph{barycentric subdivision} $B_P$ of the polyhedral map $P$ is the 3-coloured map obtained from $P$ by adding a vertex in every face of $P$
and on every edge of $P$, and adding an edge between every vertex in a face and the vertices on the boundary of this face such that every face in $B_P$ is a
triangle. The original vertices of $P$ get colour 0, the vertices corresponding to edges of $P$ get colour 1, and the vertices corresponding to faces of $P$ get
colour 2. These colours refer to the dimensions of the corresponding parts of $P$. Every face of $B_P$ has exactly one vertex of each colour. We call such a
face a \emph{chamber}. In places where more than one map is considered, we write $P$-chamber for a chamber in $B_P$. Automorphisms of coloured maps
preserve colours. In an equivalent, but purely combinatorial way, the chambers are sometimes also called {\em flags} and defined as triples $(v,e,f)$, so that
the vertex $v$ is incident with the edge $e$ that is again incident with the face $f$. The correspondence with chambers of $B_P$ is obvious.

In \cite{lspgocox} a general definition of local symmetry-preserving operations is introduced. It includes all classical and individually defined operations, such as
dual, truncation, ambo, $\dots$ and allows to prove results for all such operations together. While \cite{lspgocox} also aims at a non-mathematical audience and uses
a more intuitive and geometric definition citing a construction of Goldberg, we will use the more general and more combinatorial definition from \cite{lopsp_polyhedrality}.
This definition has the advantage to not rely on the knowledge of periodic tilings of the plane. Knowing about periodic tilings of the plane, one can think of an lsp-operation as a
triangle that is cut out of the barycentric subdivision of a tiling of the plane such that its edges are on certain symmetry axes of the tiling.

\begin{definition}\label{def:lsp}
	Let $O$ be a 2-connected plane map with vertex set $V$, together with a colouring $c: V \rightarrow \{0,1,2\}$. One of the faces is called
        the outer face. This face contains three special vertices marked as $v_0$, $v_1$, and $v_2$. We say that a vertex $v$ has \emph{colour} $i$ if $c(v)=i$.
        This 3-coloured map $O$ is a \emph{local symmetry preserving operation}, lsp-operation for short, if the following properties hold:
	\begin{enumerate}
		\item Every inner face --- i.e.\ every face that is not the outer face --- is a triangle and is called a chamber.
		\item There are no edges between vertices of the same colour.
		\item For each vertex that is not in the outer face:
		\begin{align*}
			c(v)=1 &\Rightarrow deg(v)=4\\
		\end{align*}
		For each vertex $v$ in the outer face, different from $v_0$, $v_1$, and $v_2$:
		\begin{align*}
			c(v)=1 &\Rightarrow deg(v)=3\\
		\end{align*}
		and
		\[c(v_0),c(v_2)\neq 1\]
		\begin{align*}
			c(v_1)=1 &\Rightarrow deg(v_1)=2\\
		\end{align*}
	\end{enumerate}
\end{definition}

\begin{figure}
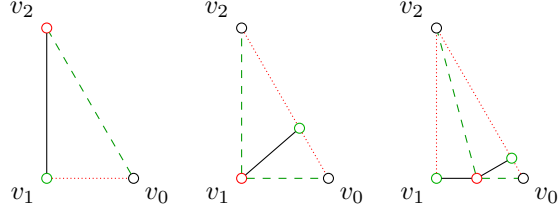

  \centering
    \input{deco_dual.tikz}
    \input{deco_ambo.tikz}
    \input{deco_truncation.tikz}
    \caption{From left to right: the c3-lsp-operations dual (from now on denoted as $D$), ambo (denoted as $A$) and truncation (denoted as $T$). Edges are assigned the unique colour that none of the
      end vertices has. The colours $0,1,2$ of vertices and edges are in this and the following figures represented as red, green, black in this order. So a red vertex is a vertex of colour $0$.}
    \label{fig:DAT}
\end{figure}

To apply an lsp-operation $O$ to a polyhedral map $P$, first take the barycentric subdivision $B_P$ of $P$. Then -- depending on the orientation of the chamber
-- a copy of $O$ or the mirror image of $O$ is glued into each chamber, identifying each vertex of colour $i$ with a copy of $v_i$ and replacing each edge
between vertices of colours $i$ and $j$ by a copy of the path between $v_i$ and $v_j$ in the outer face of $O$. The result of this gluing is a 3-coloured map
that is the barycentric subdivision of a map $O(P)$ \cite{lopsp_polyhedrality}, which is the \emph{result} $O(P)$ of applying $O$ to $P$.  Lsp-operations do not
change the genus of the polyhedral map they are applied to, as a disc is just replaced by another, subdivided, disc. In general the result of applying an
lsp-operation to a polyhedral map need not be polyhedral, e.g. if the operation has an internal $2$-cut. In \cite{lopsp_polyhedrality} it is proven that the
result of applying an lsp-operation $O$ to a specific polyhedral map $P$ is polyhedral if and only if for all polyhedral maps $P'$ the result $O(P')$ is
polyhedral. These operations are called {\em c3-lsp-operations}. So all operations transforming polyhedra to polyhedra, especially all
well known and relevant lsp-operations (e.g. the ones named by Kepler or Conway) are
c3-lsp-operations. Using the approach via tilings again, one can think of a c3-lsp-operation as a
triangle that is cut out of the barycentric subdivision of a 3-connected tiling of the plane such that its edges are on certain symmetry axes of the tiling.

For examples of c3-lsp-operations see Figure~\ref{fig:DAT}.

The \emph{inflation factor} of a c3-lsp-operation
$O$ is the ratio between the number of edges after applying the operation $O$ and the number of edges before  applying $O$.
This is equal to the ratio between the numbers of chambers after and before the operation and therefore equal to the number of chambers in $O$ \cite{lspgocox}.

\begin{definition}
  Let $x$ be a vertex, edge or chamber of $B_{O(P)}$. Then $x$ is a copy of a vertex, edge or chamber $y$ of $O$.  Let $\pi$ be the
    map that maps $x$ to $y$. The set of vertices, edges or chambers of $B_{O(P)}$ mapped to $\pi(x)$ is called the \emph{class} $\pi(x)$. 
\end{definition}

If an automorphism $\varphi$ of $B_P$ maps a chamber $C$ to a chamber $C'$, then the function mapping a chamber $C_O$ in the copy glued into $C$ to
the chamber $C'_O$ of the same class in the copy glued into $C'$ defines an automorphism of $B_{O(P)}$, that we call the {\em induced} automorphism.
Obviously all induced automorphisms are different.
 
\begin{definition}
    Let $P$ be a polyhedral map and let $O$ be a c3-lsp-operation. The submap $B^S_{O(P)}$ of $B_{O(P)}$ is the submap of $B_{O(P)}$ that consists of all the
    vertices and edges that are mapped to vertices and edges in the outer face of $O$ by $\pi$. $B^S_{O(P)}$ is isomorphic to a subdivision of $B_P$.
\end{definition}

The symmetry (or automorphism) group of a (possibly coloured) map $P$ is denoted as $\mathrm{Aut}(P)$. We consider the automorphisms -- that is: the elements of
the group -- as permutations of the set of vertices with an obvious induced action on the set of edges and faces of $P$.

There is a natural isomorphism between the groups $\mathrm{Aut}(P)$ and $\mathrm{Aut}(B_P)$ mapping a permutation $\varphi \in \mathrm{Aut}(P)$ to a
permutation $\varphi_B\in \mathrm{Aut}(B_P)$ that maps a vertex $v$ with colour $0$ to $\varphi(v)$, a vertex representing the edge $e$ to the vertex
representing the edge $\varphi(e)$ and a vertex representing the face $f$ to the vertex
representing the face $\varphi(f)$.
In most proofs we will work with  $\mathrm{Aut}(B_P)$.
As $B^S_{O(P)}$ is just $B_P$ with its edges subdivided -- and edges between
vertices of the same colour subdivided in the same way -- $\varphi \in \mathrm{Aut}(P)$ also induces an automorphism of $B^S_{O(P)}$ and therefore of $B_{O(P)}$ and $O(P)$, showing that 
$|\mathrm{Aut}(P)|\leq |\mathrm{Aut}(O(P))|$.  Given a c3-lsp-operation $O$ and a polyhedral map $P$ such that $|\mathrm{Aut}(P)| < |\mathrm{Aut}(O(P))|$, we
say that $O$ \emph{increases the symmetry} of $P$. So $O$ \emph{increases the symmetry} of $P$ if and only if there is an automorphism of $O(P)$ that is not induced
by an automorphism of $P$.
We say that a c3-lsp-operation $O$ \emph{can increase symmetry in genus $g$} if there exists a polyhedral map
of genus $g$ such that $O$ increases the symmetry of $P$.

\section{General results}

In this section we will show some general results that will be used in the rest of the paper.  Obviously the operation dual exactly preserves the
symmetries of a polyhedral map.  Interpreted as acting on the barycentric subdivision, it just interchanges the colours $0$
and $2$, so each automorphism of the dual is also an automorphism of the original map. However, as the dual is its own inverse and it preserves symmetries,
if the dual map had a larger automorphism group than the original map, then taking the dual again would imply that the original map has a larger automorphism group than itself.

Lemma~\ref{lem:dual_operations} implies that when studying which
c3-lsp-operations can increase the symmetry of polyhedral maps of a certain genus, it is sufficient to decide this question for either the operation itself or
an arbitrary combination with the operation dual: one can increase the symmetry on this genus if and only if the other can. C3-lsp-operations preserve polyhedrality \cite{lopsp_polyhedrality}, so
that the dual of a polyhedral map is also a polyhedral map.

For two operations $O,O'$ we write $(O\circ O')$ for the operation that transforms a map $P$ into the map $O(O'(P))$. The operation $(O\circ O')$ can also be
described as -- similar to applying it to a map that already has a barycentric subdivision -- gluing a copy of $O$ or its mirror image into every chamber of $O'$, ignoring the outer face. 

\begin{lemma}\label{lem:dual_operations}
  Let $O$ be a c3-lsp-operation and $P$ a polyhedral map. Then the following three statements are equivalent: 

  \begin{enumerate}

   \item $|\mathrm{Aut}(O(P))| > |\mathrm{Aut}(P)|$ (that is: $O$ increases the symmetry of $P$)

  \item $|\mathrm{Aut}((D \circ O)(P))| > |\mathrm{Aut}(P)|$ (that is: $D \circ O$ increases the symmetry of $P$)

  \item $|\mathrm{Aut}((O \circ D)(D(P)))| > |\mathrm{Aut}(D(P))|$ (that is: $O \circ D$ increases the symmetry of $D(P)$).

    \end{enumerate}

\end{lemma}

\begin{proof} 
  $1. \Leftrightarrow 2.$: This is immediate as $|\mathrm{Aut}(O(P))| = |\mathrm{Aut}((D \circ O)(P))|$

  $1. \Leftrightarrow 3.$:  As $D\circ D$ is the identity operation, we have that $|\mathrm{Aut}((O \circ D)(D(P)))| = |\mathrm{Aut}(O(P))|$
  and as $ |\mathrm{Aut}(D(P))|=|\mathrm{Aut}(P)|$ we get the equivalence.

\end{proof}

The following two lemmas are well known, but mentioned for completeness and later use.

\begin{lemma}\label{lem:facesize5_plane}
    Every (plane) polyhedron has a face of size at most 5.
\end{lemma}

\begin{lemma}\label{lem: the action on the chambers is free}
  Let $P$ be a polyhedral map.  Then $\mathrm{Aut}(P)=\mathrm{Aut}(B_P)$ and
 $\mathrm{Aut}(P)$ acts freely on the set of chambers of $B_P$, so the image of a single chamber determines the whole automorphism.
\end{lemma}

\begin{corollary}\label{cor:sym_uncol_B}
  Let $P$ be a polyhedral map, $B^u_P$ the barycentric subdivision with the colours removed and
  $\mathrm{S_i}(B_P) $ the set of automorphisms of $B^u_P$ interchanging the former colours $0$ and $2$ and respecting colour $1$ (which can be interpreted as maps onto the dual of $P$).
  
  Then (with $\mathrm{Aut}(B_P)$ also considered just as a set) the elements of the group $\mathrm{Aut}(B^u_P) $ are exactly the permutations in
  $\mathrm{Aut}(B_P) \cup \mathrm{S_i}(B_P)$.
\end{corollary}

\begin{proof}
  It is enough to show that an automorphism $\varphi$ of $B^u_P$ that is not in $\mathrm{Aut}(B_P)$ interchanges vertices of former colour $0$ and $2$ and respects colour
  $1$ and therefore is in $\mathrm{S_i}(B_P)$.

  As $P$ is 3-connected, every vertex in $B_P$ of colour $0$ or $2$ has degree at least $6$. Vertices of colour $1$ always have degree $4$. Therefore $\varphi$
  sends vertices of colour $1$ to vertices of colour $1$. As $\varphi$ does not preserve colours, it needs to send at least one vertex of colour $0$ to a vertex
  of colour $2$. $B_P$ is a connected tri-partite map, with partitions given by the colours. In fact there is a path between any two vertices of colour $0$ or
  $2$ not containing a vertex of colour $1$. If $\varphi$ sends a vertex of colour $0$ to one of colour $2$, then it interchanges colours along all such paths,
  so all vertices of colour $0$ are sent to vertices of colour $2$. Therefore, $P$ is self-dual and this proves the statement.
\end{proof}

From the previous proof we obtain that $|\mathrm{Aut}(P) |< |\mathrm{Aut}(B^u_P)|$ if and only if $P$ is self-dual, because there is an automorphism of $B^u_P$
which sends each vertex of $P$ to a vertex of the dual.

\begin{theorem}\label{thm: automorphism preserving subdivision is ambo}
    Let $P$ be a polyhedral map and let $O$ be a c3-lsp-operation such that every automorphism in $\mathrm{Aut}(B_{O(P)})$ maps $B^S_{O(P)}$ to $B^S_{O(P)}$. If
    $O$ increases the symmetry of $P$, then $P$ is a self-dual polyhedral map and $O= X\circ A$, with $A$ the ambo operation and $X$ a c3-lsp-operation.
\end{theorem}

\begin{proof}
    Let $\varphi$ be an automorphism of $B_{O(P)}$ which is not induced by an automorphism of $P$. As an automorphism of $B_{O(P)}$ it is colour-preserving. As
    $\varphi$ maps $B^S_{O(P)}$ to $B^S_{O(P)}$, it induces an automorphism $\hat{\varphi}$ of $B^u_P$. This automorphism $\hat{\varphi}$ sends each vertex of
    colour $0$ to a vertex of colour $2$ by \Cref{cor:sym_uncol_B}. As the barycentric subdivision of the dual of $P$ is just $B_P$ with the colours $0$ and $2$
    switched, $\hat{\varphi}$ can be seen as an isomorphism between $P$ and its dual. It follows that $P$ is self-dual.
    
    As an automorphism of $B_{O(P)}$, $\varphi$ maps $B^S_{O(P)}$ to $B^S_{O(P)}$. Our previous argument implies that the boundary of a face in $B^S_{O(P)}$ is
    mapped to the boundary of a face in $B^S_{O(P)}$ in such a way that the vertices corresponding to vertices of colour 0 in $B_P$ are mapped to vertices
    corresponding to vertices of colour 2 in $B_P$ and the other way around. Vertices that are of colour 1 in $B_P$ are not mapped to another colour. This
    implies that $\varphi$ maps each copy of $O$ to a mirrored copy of $O$. More specifically, $O$ must be mirror symmetric with respect to a path joining $v_1$
    with the {\em midpoint} between $v_0$ and $v_2$ in the outer face. The boundary path between $v_0$ and $v_2$ contains an odd number of vertices, so that a
    midpoint $v$ and also a path $M$ between $v_1$ and $v$ stabilized by the mirror symmetry exists, as no chamber can be mapped to itself by a mirror symmetry
    because of the colouring. The fact that each of the two parts into which $M$ splits the operation is in fact an lsp-operation (in fact one the mirror image of the other) follows for the vertices not on
    the boundary between the two parts directly from the fact that $O$ is an lsp-operation. The conditions for the other vertices and the special vertices can
    easily be checked. The fact that they are c3-lsp-operations follows from the fact that $O(P)$ is polyhedral for a polyhedral map $P$ and $O(P)=X(A(P))$ with $X$ the lsp-operation
    defined by one of the parts.
\end{proof}

\pagebreak[2]

\begin{lemma}\label{lem:map_chamber_to_same_class}
    Let $P$ be a polyhedral map and $O$ a c3-lsp-operation. The following are equivalent:
    \begin{enumerate}
    \item The operation $O$ increases the symmetry of $P$, that is: \\
      $|\mathrm{Aut}(B_{O(P)})|>|\mathrm{Aut}(B_P)|$.
        \item There exists an automorphism of $B_{O(P)}$ that maps a chamber to a chamber in a different class.
        \item There exists an automorphism of $B_{O(P)}$ that maps every chamber to a chamber in a different class.
    \end{enumerate}
\end{lemma}
\begin{proof}

  $1. \Rightarrow 2.$: If an automorphism $\varphi$ of $B_{O(P)}$ maps all chambers to chambers of the same class, then it induces an automorphism of the chambers of $B_P$, so
  $\varphi$ is one of the automorphisms of $B_{O(P)}$ induced by an automorphism of $\mathrm{Aut}(B_P)$. If all automorphisms of  $B_{O(P)}$ have that property, then $|\mathrm{Aut}(B_{O(P)})|=|\mathrm{Aut}(B_P)|$.

  $1. \Leftarrow 2.$:  If an automorphism of $B_{O(P)}$ maps a chamber to a chamber in a different class, then it is none of the induced automorphisms, so $|\mathrm{Aut}(B_{O(P)})|>|\mathrm{Aut}(B_P)|$.

  $2. \Rightarrow 3.$: If an automorphism maps a chamber $C$ to a chamber in the same class, then all chambers sharing an edge with $C$ are mapped to
  chambers in the same class and by induction all chambers in $B_{O(P)}$ have this property.

  $2. \Leftarrow 3.$: Immediate.
\end{proof}

\section{Goldberg-Coxeter operations}

In chemistry and biology, 3-regular polyhedra that only have faces of size five and six are also known as fullerenes. They are often studied for their
interesting chemical properties and technological applications. The most common application of Goldberg-Coxeter operations is the construction of
fullerenes. The first publication describing these operations was by Goldberg in 1937 \cite{goldberg1937} in a mathematical context. Later, in 1962, closely
related constructions with the same resulting structures were described by Caspar and Klug \cite{caspar1962}, this time in a biological context as models for virus capsids (that is: protein
shells). These were later also described in a survey paper by Coxeter \cite{Coxeter_virus}. For a detailed description of the history of Goldberg-Coxeter operations we refer the reader to \cite{lspgocox}.

We will follow the approach in \cite{lspgocox} and define Goldberg-Coxeter operations, GC-operations for short, as triangles cut out of the barycentric
subdivision of the regular hexagonal tiling $T_H$ of the plane. Some of the statements in this paragraph are taken from that article. We use the following
coordinate system to describe $T_H$. The origin $(0,0)$ is in the middle of a face $f$. One vertex of $f$ is $(0,1)$ and the vertex of $f$ that appears in the
boundary of $f$ right before $(0,1)$ when following $\partial f$ in clockwise order is $(1,0)$. With this coordinate system, the point with integer coordinates $(x,y)$
is the center of a face of $T_H$ if and only if $x-y\equiv 0 \pmod{3}$, and a vertex otherwise. The point $(x,y)$ is the middle of an edge if and only if $x$
and $y$ are not both integers, but they are multiples of $1/2$ and $2(x-y)\equiv 0 \pmod{3}$.

Let $l$ and $m$ be two positive integers such that $l=m$ or $m=0$. The \emph{Goldberg-Coxeter operation} with parameters $(l,m)$ (short $\mathrm{GC}(l,m)$) is the labeled map
that is obtained by cutting a triangle out of $B_{T_H}$ with vertices $v_0=(l,m)$, $v_1=\left(\frac{l-m}{2}, \frac{l+2m}{2}\right)$ and $v_2=(0,0)$. The
point $v_1$ is the middle of the line segment between $v_0$ and the image of $v_0$ under a $60^\circ$ counterclockwise rotation around the origin. Examples of
GC-operations $\mathrm{GC}(l,0)$ and $\mathrm{GC}(l,l)$ are given in Figure~\ref{fig:GC_examples}.

GC-operations are also defined for non-negative integer parameters $(l,m)$ not satisfying the extra conditions $l=m$ or $m=0$ imposed here. These operations are known as
chiral GC-operations and do not necessarily preserve all the symmetries of a polyhedron, but only the orientation-preserving ones. Such GC-operations can be described
as c3-lopsp-operations -- see \cite{lopsp_polyhedrality} for a definition of c3-lopsp-operations. In this article we only consider GC-operations preserving all
symmetries. 

\begin{figure}
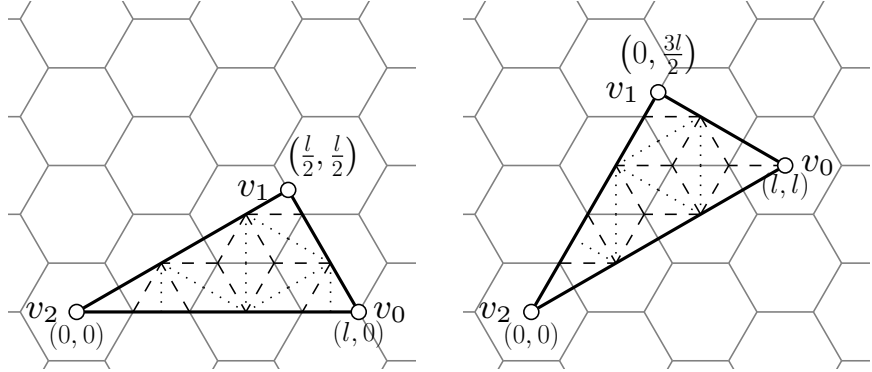

    \centering
    \scalebox{1.5}{\input{hexagontiling_m0.tikz}}
    \quad
    \scalebox{1.5}{\input{hexagontiling_mm.tikz}}
    \caption{The left image shows the Goldberg-Coxeter operation $\mathrm{GC}(5,0)$, and the right image shows the Goldberg-Coxeter operation $\mathrm{GC}(3,3)$.}
    \label{fig:GC_examples}
\end{figure}

We have the following lemma:

\begin{lemma}\label{lem:vi_degrees}
    Let $O$ be a GC-operation $\mathrm{GC}(l,l)$ or $\mathrm{GC}(l,0)$. Then it has the following properties:
    \begin{itemize}
        \item $v_2$ is in only one chamber of $O$
        \item If $v_0$ has colour 0, then $v_0$ is in only one chamber of $O$. Otherwise it is in two chambers.
    \end{itemize}
\end{lemma}
\begin{proof}
  
    \begin{itemize}
        \item The angle at $v_2$ in the triangle cut out of $B_{T_H}$ is always $30^\circ$. In $B_{T_H}$, the 12 edges incident with each vertex of colour 2,
          i.e. the faces, form angles of $30^\circ$ with their successors and predecessors. Therefore, $v_2$ is in exactly one chamber of $O$.
        \item The angle at $v_0$ in the triangle cut out of $B_{T_H}$ is always $60^\circ$. If $v_0$ corresponds to a face, it follows from the previous
          argument that it is in exactly two chambers of $O$. Every vertex of colour 0 in $B_{T_H}$ has degree 6 and the incident edges form angles of
          $60^\circ$. Therefore, if $v_0$ has colour 0, then it is in only one chamber of $O$.
    \end{itemize}
\end{proof}

Applying a GC-operation to $T_H$ results in a regular hexagonal tiling with smaller hexagons. In this infinite case 
not only the symmetry group of the result is the same as before applying the operation, but as there is up to isomorphism only one hexagonal tiling of the plane, even
the tiling is the same (up to isomorphism). This is not the case when we apply a GC-operation that is
not the identity to a finite $3$-regular map with only hexagonal faces. The result is another 3-regular map with only hexagonal faces, but the map and also the symmetry group are larger.
We use the following result by Negami  to prove that every non-trivial GC-operation increases the symmetry of every map on the torus with only hexagons as faces.

\begin{lemma}[S. Negami \cite{negami1983uniqueness}]\label{lem:negami}
    Every (simple) 6-regular map of genus 1 is vertex-transitive.
\end{lemma}

\begin{lemma}\label{lem:GC_genus1}
  Let $P$ be any polyhedral map of genus $1$ that has only faces of size $6$. Then any GC-operation $\mathrm{GC}(l,l)$ or $\mathrm{GC}(l,0)$
  that is not the identity (that is: $\mathrm{GC}(1,0)$) increases the symmetry of $P$.
\end{lemma}

\begin{proof}
    It is not difficult to prove using the Euler characteristic that any map of genus 1 that only has faces of size 6 is 3-regular. If such a map is also polyhedral, its dual is a simple, 6-regular map of genus 1. By Lemma~\ref{lem:negami}, it is vertex-transitive. Therefore any polyhedral map of genus 1 that
    only has faces of size 6 is face-transitive.

    Let $O$ be any GC-operation that is not the identity. If follows from the definition of GC-operations that all faces of $O(P)$ have size 6 and therefore it
    is face-transitive. The vertex $v_2$ is of colour 2 and it is in only one chamber $C$ in $O$ by Lemma~\ref{lem:vi_degrees}. This means that for every face in
    $P$, there is exactly one face in $O(P)$ that consists entirely of chambers of class $C$. All the chambers of class $C$ are in such faces. As there is at
    least one other chamber in $O$, there is a face in $O(P)$ that contains no chambers of class $C$. However, as $O(P)$ is face-transitive, this implies that
    there is an automorphism that maps a chamber of class $C$ to a chamber of another class. By Lemma~\ref{lem:map_chamber_to_same_class}, $O$ increases the
    symmetry of $P$.
\end{proof}

\subsection{Truncation and the GC-operation $\mathrm{GC}(1,1)$}\label{subsec:GC_bitruncation}

The only GC-operation that we will consider by itself is $\mathrm{GC}(1,1)$, also known as {\em bitruncation}, {\em leapfrog}, and {\em
  zip}. We will use the results in this section to determine when other GC-operations can increase symmetry. As the name suggests, bitruncation is closely
related to truncation, or to be exact: bitruncation is the truncation of the dual map. Both c3-lsp-operations are shown in
Figure~\ref{fig:truncation}. Informally, truncation `cuts off' the
vertices of a polyhedron, replacing a vertex $v$ by a cycle of length $deg(v)$. Bitruncation will be denoted by $B$ and truncation by
$T$. We will prove that truncation cannot increase the symmetry of a polyhedron, but it can increase the symmetry for polyhedral maps of higher genus. As
$B =T\circ D$, Lemma~\ref{lem:dual_operations} implies that that is also true for bitruncation.

\begin{figure}
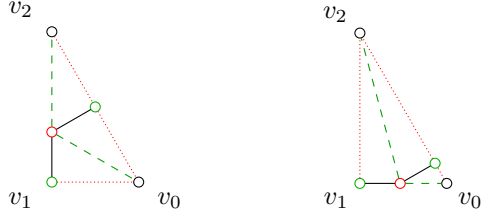

    \centering
    \scalebox{1}{\input{deco_bitruncation.tikz}}
    \qquad\qquad
    \scalebox{1}{\input{deco_truncation.tikz}}
    \caption{The operation bitruncation is shown on the left, truncation is on the right.}
    \label{fig:truncation}
\end{figure}

\begin{lemma}\label{lem: truncation does not increase symmetry on plane}
    Let $P$ be a (plane) polyhedron and $T$ the c3-lsp-operation truncation. Then \[\mathrm{Aut}(P)= \mathrm{Aut}(T(P)).\] 
\end{lemma}
\begin{proof}
  Let $\varphi$ be any automorphism of $T(P)$. As $T(P)$ is plane, there exists a face in $T(P)$ of size at most 5 by Lemma~\ref{lem:facesize5_plane} and a
  corresponding vertex $f$ in $B_{T(P)}$ of degree at most $10$. The only vertices of colour 2 in $T$ are $v_0$ and $v_2$, so $f$ is mapped to $v_0$ or $v_2$ by
  $\pi$. However, as every face of $P$ has size at least $3$, the colour 2 vertices of $T(P)$ that are mapped to $v_2$ by $\pi$ must have degree at least $12$.
  Therefore $\pi(f) =  v_0$. As automorphisms map vertices to vertices of the same degree, $\pi(\varphi(f))$ is also $v_0$. Let $C$ be a $T(P)$-chamber containing $f$. Then $C$ and
  $\varphi(C)$ both contain a vertex in $\pi^{-1}(v_0)$. As $v_0$ is in only one chamber in $T$, the $T(P)$-chambers $C$ and $\varphi(C)$ are in the same
  class. The lemma now follows from Lemma~\ref{lem:map_chamber_to_same_class}.
\end{proof}

We now define a class $\{H_g\mid g\in \mathbb{N}\setminus\{0,1\}\}$ of polyhedral maps. Each polyhedral map $H_g$ will have genus $g$, and $T(H_g)$ will have
strictly more automorphisms than $H_g$. In \cite{korbitmaps} the map $H_2$ -- a polyhedral map of genus $2$ -- is defined and it is shown that truncation
increases its symmetry. That example can be extended to higher genera. Figure~\ref{fig:genus3_cut} shows for $g=3$ how
the polyhedral map $H_g$ of genus $g$ is constructed from the polyhedral map $H_{g-1}$ of genus $g-1$. The map $H_2$ is the map that is
obtained by ignoring the slices with orange and red arrows, and gluing the dashed lines together. The map $H_3$ is obtained by inserting the slices with red and
orange arrows into $H_2$ as shown, gluing dashed lines together. For larger $g$ this process can be iterated.

More formally, $H_g$ is defined as follows. The map $H_g$ has $4g$ vertices \linebreak[4] $A_0,\ldots, A_{2g-1},B_0,\ldots,B_{2g-1}$. The rotation system is
defined as follows -- where indices are taken modulo $2g$:

\begin{center}
\begin{tabular}{llllllll}
    $A_{i}$:&$A_{i+1}$&$A_{i-1}$&$B_{i}$&$B_{i+g}$&$B_{i+g+1}$&$B_{i+1}$\\
    $B_{i}$:&$B_{i-1}$&$B_{i+1}$&$A_{i+g}$&$A_{i}$&$A_{i-1}$&$A_{i+g-1}$
\end{tabular}
\end{center}

The map $H_g$ has two faces of size $2g$, $2g$ faces of size 4, and $4g$ faces of size 3. Every face of size $2g$ is adjacent to $2g$ different triangles, every
quadrangle is adjacent to four different triangles, and every triangle is adjacent to one face of size $2g$ and two different quadrangles.  The top image in
Figure~\ref{fig:genus3_orbits} shows $H_3$.

\begin{figure}
    \centering
    \includegraphics[width =0.7\textwidth]{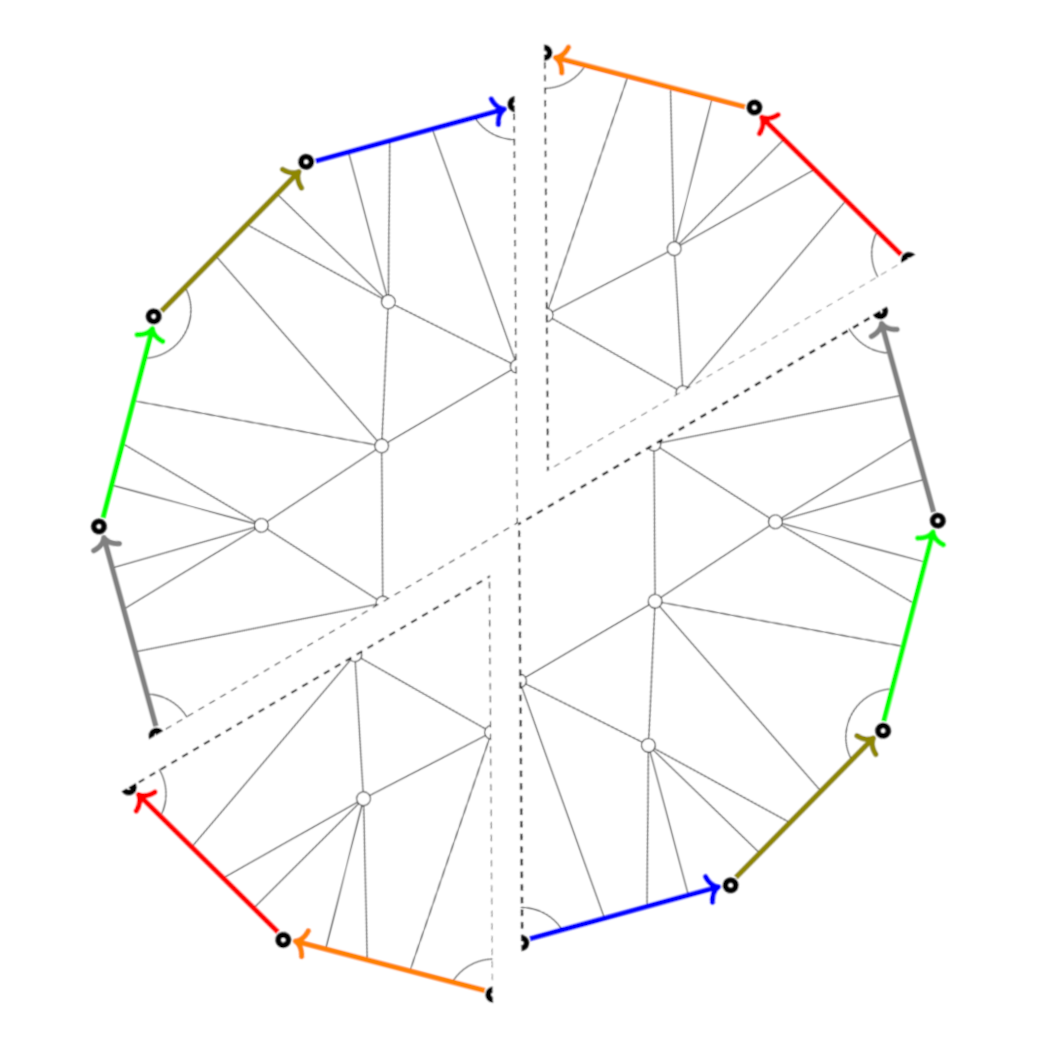}
    \caption{This figure shows how $H_3$ is constructed from $H_2$. Arrows with the same colour must be identified.}
    \label{fig:genus3_cut}
\end{figure}

\begin{figure}
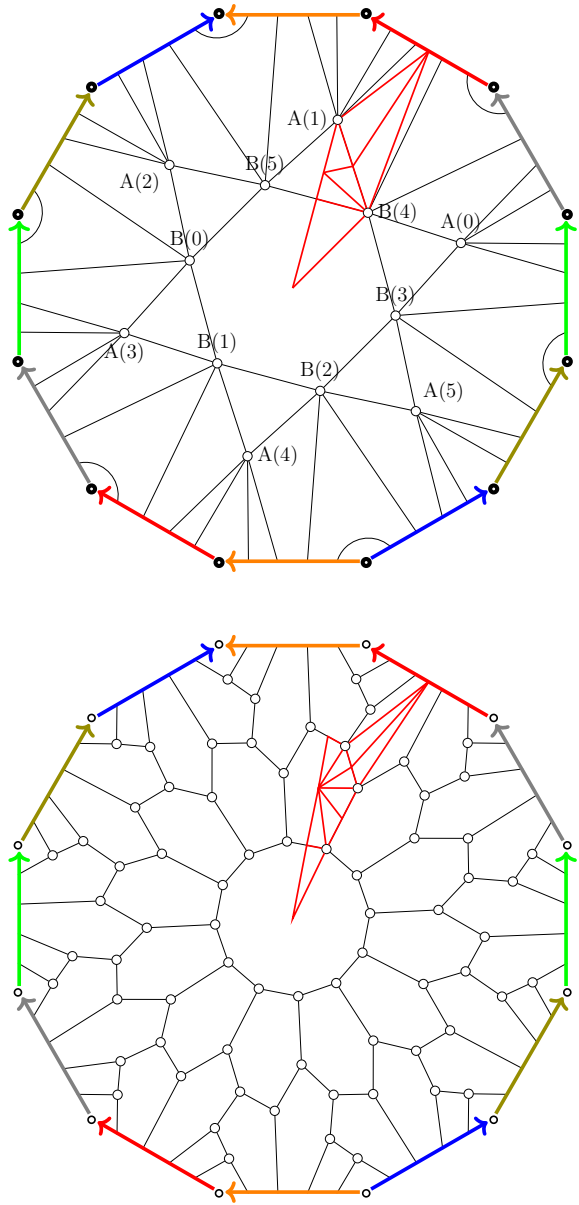

    \centering
    \scalebox{0.54}{\input{genus3_withorbits.tikz}}
    
    \vspace*{0.8cm}
    \scalebox{0.54}{\input{genus3trunc_withorbits.tikz}}
    \caption{The maps $H_3$ and $T(H_3)$, each with one chamber from each orbit drawn in red.}
    \label{fig:genus3_orbits}
\end{figure}

\begin{lemma}\label{lem:truncationfrom1}
    For every $g\in \mathbb{N}\setminus \{0\}$, there exists a polyhedral map $P$ such that truncation increases the symmetry of $P$.
\end{lemma}
\begin{proof}
    For $g=1$ this follows from Lemma~\ref{lem:GC_genus1}.

    For $g=2$, it is stated in \cite{korbitmaps} that $H_2$ has $2$ chamber-orbits and its truncation has $3$ chamber-orbits. Truncation triples the number of
    chambers, but the number of orbits of $T(H_2)$ is only $3/2$ times the number of orbits of $H_2$. It follows from
    \Cref{lem: the action on the chambers is free} that for any polyhedral map, the number of elements in a chamber-orbit equals the size of the symmetry group. If $n_P$
    represents the number of chamber-orbits in a polyhedral map $P$ and $\mathcal{C}^{P}$ the number of chambers in $P$, then
    
    \[\dfrac{|\mathrm{Aut}(T(H_2))|}{|\mathrm{Aut}(H_2)|}=\dfrac{\dfrac{\mathcal{C}^{T(H_2)}}{n_{T(H_2)}}}{\dfrac{\mathcal{C}^{H_2}}{n_{H_2}}}= \dfrac{\mathcal{C}^{T(H_2)}}{\mathcal{C}^{H_2}}\cdot \dfrac{n_{H_2}}{n_{T(H_2)}} =3\cdot \dfrac{2}{3}=2.\]

    For $g> 2$, it can be checked that the map $H_g$ has 6 chamber-orbits, one orbit consists of chambers in the faces of size $2g$, 3 orbits of chambers in the triangles, and 2 orbits of chambers in the quadrangles. The map $T(H_g)$ has 9 chamber-orbits, 6 in the faces of size $6$, one in faces of size $4g$ and two in the faces of size 8. For $g=3$, the chamber-orbits are shown in Figure~\ref{fig:genus3_orbits}. As in the case $g=2$, we get $\frac{|\mathrm{Aut}(T(H_g))|}{|\mathrm{Aut}(H_g)|}=2$. 
\end{proof}

\begin{corollary}\label{cor:bitruncation}
  There is a polyhedral map $P$ of genus $g$ so that truncation increases the symmetry of $P$ if and only if $g>0$.

  As $\mathrm{GC}(1,1)=T\circ D$, the same is true for $\mathrm{GC}(1,1)$.
\end{corollary}

\subsection{GC-operations where $v_0$ has colour 0}\label{subsec:GC_v0colour0}

In this section we look at GC-operations where $v_0$ has colour 0, i.e. where $(l,m)$ is a vertex in $T_H$.  We have seen that this is the case if
$x-y\not\equiv 0 \pmod{3}$ and thus $l\neq m$. It follows that $m=0$. We have already proven in Lemma~\ref{lem:GC_genus1} that for genus $1$ all GC-operations can increase
symmetry. In Lemma~\ref{lem:GC_colour0_allgenus} we will prove that for GC-operations with $v_0$ of colour 0, genus 1 is the only genus where they can increase
symmetry.

\begin{lemma}\label{lem:GC_colour0_allgenus}
    Let $P$ be a polyhedral map of a genus $g\not= 1$ and let $O$ be a GC-operation with parameters $(l,0)$ such that $l$ is not a multiple of 3. Then $O$ does not increase the symmetry of $P$.
\end{lemma} 
\begin{proof}
    Let $\varphi$ be any automorphism of $O(P)$. By Lemma~\ref{lem:map_chamber_to_same_class} it suffices to show that there is a chamber $C$ such that $\varphi(C)$ is in the same class as $C$.
    As the genus is not 1, $O(P)$ has a vertex of degree at least 4 or a face of size different from 6. 
    Assume first that there is a vertex $v$ in $O(P)$ with $deg(v) \neq 3$. 
    Every vertex $w$ in $O(P)$ such that $\pi(w) \neq v_0$, i.e. $w$ does not correspond to a vertex of $P$, has degree 3. Therefore, $\pi(v)=v_0$ and as $deg(v) = deg(\varphi(v))$ also $\pi(\varphi(v))=v_0$. Let $C$ be a chamber in $O(P)$ containing $v$. The chamber $\varphi(C)$ contains $\varphi(v)$. It follows from Lemma~\ref{lem:vi_degrees} that $v_0$ is incident to only one chamber in $O$. Therefore, $C$ and $\varphi(C)$ are in the same class. 

    Now assume that there is a face $f$ in $O(P)$ that is not of size 6. As every face in $O(P)$ that does not correspond to a face of $P$ has size 6, $\varphi(f)$ corresponds to a face of $P$. Similarly as in the previous case, Lemma~\ref{lem:vi_degrees} implies that there is exactly one chamber $C$ containing $f$ such that $C$ and $\varphi(C)$ are in the same class.
\end{proof}

\subsection{GC-operations where $v_0$ has colour 2}\label{subsec:GC_v0colour2}

In Subsection~\ref{subsec:GC_bitruncation} it was proven that bitruncation and truncation can increase symmetry on every genus except genus 0. As for
bitruncation $v_0$ is of colour 2, it is clear that the results from Subsection~\ref{subsec:GC_v0colour0} do not hold if $v_0$ is of colour 2. In
Corollary~\ref{cor:GC_v0colour2} it will be proven that, just like bitruncation, every GC-operation with $v_0$ of colour 2 can increase symmetry on every genus
except genus 0.

\begin{lemma}\label{lem:GC_decomposition}
  ~
  
  \begin{itemize}
 
  \item If $\mathrm{GC}(l,0)$ is a GC-operation such that $v_0$ is of colour 2, then $l=3k$, $k\in \mathbb{N}$, and  $\mathrm{GC}(3k,0)=\mathrm{GC}(k,k)\circ \mathrm{GC}(1,1)$.

     \item For any GC-operation $\mathrm{GC}(l,l)$, $\mathrm{GC}(l,l)=\mathrm{GC}(l,0)\circ \mathrm{GC}(1,1)$.
    \end{itemize}

\end{lemma}
\begin{proof}
    We already mentioned that a point with coordinates $(x,y)$ is the center of a face of $T_H$ if and only if $x-y\equiv 0 \pmod{3}$. It follows that if $m=0$,
    then $l=3k$ for a natural number $k$.  Assume first that $m=0$ and $l=3k$ for a natural number $k$. The center of the triangle with vertices $(0,0)$,
    $(0,3k)$ and $(3k,0)$ has coordinates $(k,k)$. As $k-k\equiv 0 \pmod{3}$, this point corresponds to a face of $T_H$ and there are symmetry axis through
    $(0,0)$ and $(k,k)$, through $(0,3k)$ and $(k,k)$ and through $(3k,0)$ and $(k,k)$. As shown in Figure~\ref{fig:GC_bitruncation}, these symmetries imply
    that $O$ consists of three copies of the GC-operation $\mathrm{GC}(k,k)$ (shown in red) so that $\mathrm{GC}(3k,0)=\mathrm{GC}(k,k)\circ \mathrm{GC}(1,1)$.
    
    Now assume that $l=m\neq0$. In this case the center of the triangle with vertices $(0,0)$, $(l,l)$ and $(-l,2l)$ is $(0,l)$. This is a vertex if $l$ is not
    a multiple of 3, and a face if it is. In any case, there are mirror axis through $(0,0)$ and $(0,l)$, through $(l,l)$ and $(0,l)$ and through $(-l,2l)$ and
    $(0,l)$. This is illustrated in Figure~\ref{fig:GC_bitruncation}. It follows that $\mathrm{GC}(l,l)$ consists of three copies of the GC-operation $\mathrm{GC}(l,0)$ 
    in such a way that  $\mathrm{GC}(l,l)=\mathrm{GC}(l,0)\circ \mathrm{GC}(1,1)$.
\end{proof}

\begin{figure}
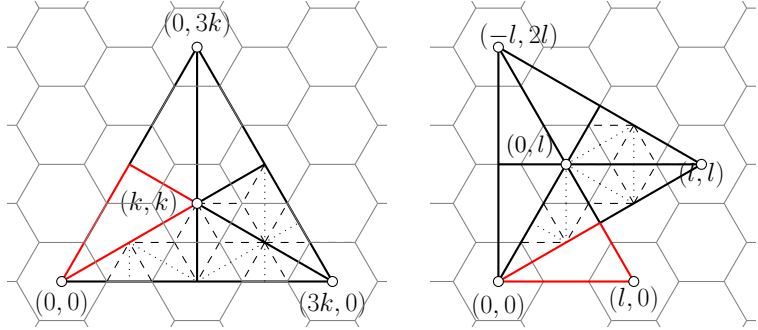

    \centering
    \scalebox{1}{\input{hexagontiling_m0_bitruncation.tikz}}
    \quad
    \scalebox{1}{\input{hexagontiling_mm_bitruncation.tikz}}
    \caption{These images show the symmetries mentioned in the proof of Lemma~\ref{lem:GC_decomposition}. The left image shows the case where $m=0$, and the right image shows the case where $l=m$.}
    \label{fig:GC_bitruncation}
\end{figure}

\begin{corollary}\label{cor:GC_v0colour2}
    Let $\mathrm{GC}(l,l)$ or $\mathrm{GC}(l,0)$ be a GC-operation where $v_0$ has colour 2. Then there exist polyhedral maps of genus $g$ for which
    $\mathrm{GC}(l,m)$ can increase the symmetry if and only if $g\neq 0$.
\end{corollary}

\begin{proof}
    We prove this by induction on the inflation factor. The smallest GC-operation with $v_0$ of colour 2 is $\mathrm{GC}(1,1)$. For this operation the result follows
    from Corollary~\ref{cor:bitruncation}. Now assume that the result is true for every GC-operation with $v_0$ of colour 2 and inflation factor at most $n-1$.
    
    Let $\mathrm{GC}(l,m)$ be a GC-operation with $l=m$ or $m=0$, with $v_0$ of colour 2, and with inflation factor $n$. By Lemma~\ref{lem:GC_decomposition} there
    exists a GC-operation $\mathrm{GC}(l',m')$ such that $\mathrm{GC}(l,m)=\mathrm{GC}(l',m')\circ \mathrm{GC}(1,1)$.

    The inflation factor of $\mathrm{GC}(1,1)$ is 3, so the inflation factor of $\mathrm{GC}(l',m')$ is $n/3 < n$.  There exist polyhedral maps of every
    genus $g>0$ for which $\mathrm{GC}(1,1)$ increases the symmetry. As $\mathrm{GC}(l',m')$ at least preserves the symmetry, the symmetry of these
    polyhedral maps is also increased by $\mathrm{GC}(l,m)$. $\mathrm{GC}(1,1)$ cannot increase symmetry in polyhedra, so it suffices to prove that
    $\mathrm{GC}(l',m')$ cannot increase symmetry in polyhedra. If the $v_0$ vertex of $\mathrm{GC}(l',m')$ is of colour 2, this follows by induction. If it
    is of colour 0, it follows from Lemma~\ref{lem:GC_colour0_allgenus}. This proves the corollary.
\end{proof}

\section{Operations with a small inflation factor}

In this part we answer the question of which operations can increase symmetry and, if so, on which genus they can do so for all c3-lsp-operations with inflation
factor at most 6. This includes all well-known Conway operations that preserve all symmetries.

The Conway polyhedron notation was introduced by John Conway \cite{conway2008symmetries} to denote polyhedra obtained from operations such as truncation or
dual. In this notation, a polyhedron is denoted by a capital letter (e.g. $T$ is a tetrahedron) and the operation applied to the polyhedron is denoted with a
lowercase letter (e.g. $t$ is truncation). For example, the truncated tetrahedron is denoted by $tT$.

The operations named by Conway and later by Hart, Rossiter and Levskaya can be described as c3-lsp-operations if they preserve all symmetries and as c3-lopsp-operations
(see \cite{lspgocox}) if they are only guaranteed to preserve orientation preserving symmetries.

\subsection{Ambo}
The Conway operation ambo plays a special role, as so far it is {\em essentially} the only operation known to be able to increase the symmetry of (plane) polyhedra. All
other known operations that do so are combinations of ambo.
Ambo applied to a polyhedron $P$ can be described as placing a vertex in the midpoint of every edge of $P$ and connecting two of these
vertices through the common face if the corresponding edges of $P$ are incident to the same vertex in $P$ and belong to the same face of $P$. This is equivalent to the
graph theoretical construction of the medial graph.  Ambo is depicted as a c3-lsp-operation in Figure~\ref{fig:DAT}.

It is known (see for example \cite{korbitmaps}) that ambo increases the symmetry of self-dual polyhedral maps and 
that it can only do so for self-dual maps.
In particular, the symmetry group of the
polyhedral map after the application of ambo will be twice as big as the original symmetry group, where a new symmetry is obtained by composing an old symmetry with a
mapping on the dual.  Since there exist self-dual polyhedra, ambo can increase symmetry in genus zero. For example, ambo of a tetrahedron is an octahedron.

To prove that ambo can increase symmetry in every genus, we looked for results stating that self-dual maps exist in every genus. Though this is probably
folklore, we found no such results in the literature. A construction for self-dual maps by Archdeacon is sketched in
\cite{archdeacon1992survey}. There it is said that the construction is described in more detail in another paper in preparation, which seems not to have appeared.
In general, the construction also does not guarantee that the result is a polyhedron, or even that it is connected.  In the rest of this
section, we prove the existence of self-dual polyhedral maps in every genus, using a special case of the construction by Archdeacon for genus $g\ge 2$.

\begin{theorem}
    There exist self-dual polyhedral maps in every genus.
\end{theorem}

\begin{figure}
    \centering
     \scalebox{0.8}{\input{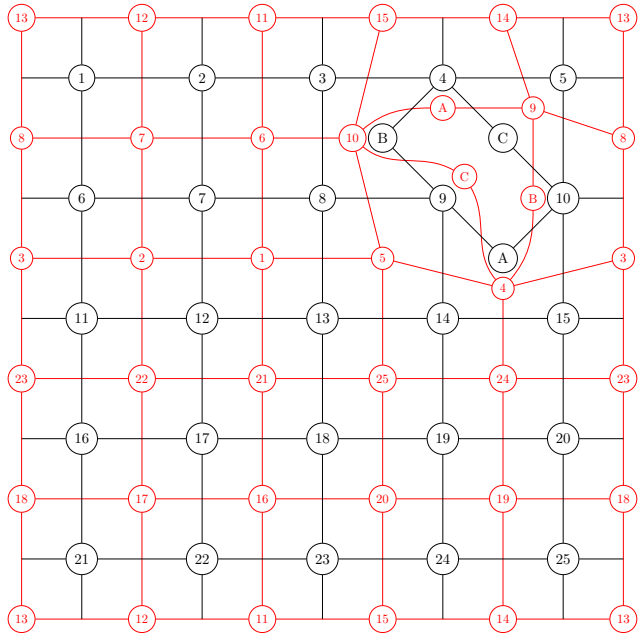}}
    \caption{The maps $\mathcal{G}$ (in black) and $\mathcal{G}'$ (in red) on the torus.}
    \label{fig: square tiling torus with 1 hexagon}
\end{figure}

\begin{figure}
    \centering
      \scalebox{0.8}{\input{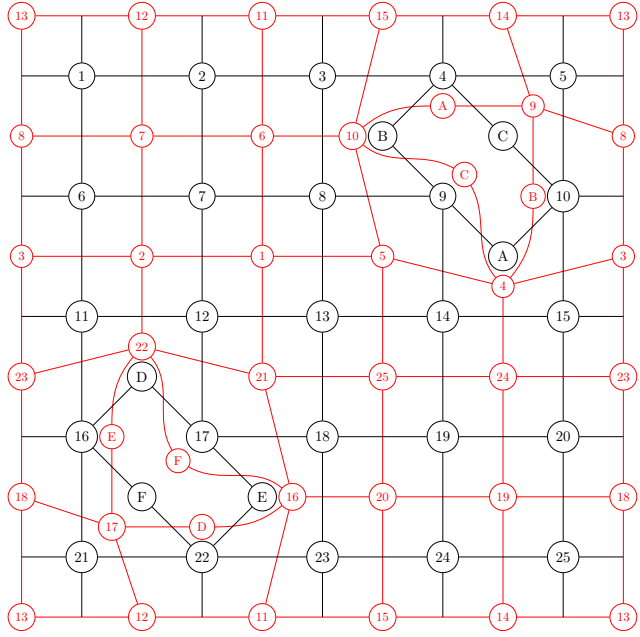}}
    \caption{The maps $\mathcal{H}$ (in black) and $\mathcal{H}'$ (in red) on the torus.}
    \label{fig: square tiling torus with 2 hexagons}
\end{figure}

\begin{proof}
    In genus 0, the tetrahedron is a self-dual polyhedron. In genus 1, the square tiling of a torus gives a self-dual polyhedral map.

    For genus 2, consider the square tiling of the torus with three squares replaced by a hexagon, a square, and two pentagons. This map $\mathcal{G}$ is depicted in black
    in Figure~\ref{fig: square tiling torus with 1 hexagon}. The tiling depicted in red in Figure~\ref{fig: square tiling torus
      with 1 hexagon} is a map isomorphic to $\mathcal{G}$, where the isomorphism $\psi$ is given by sending vertices with the same label to each other. Let
    $\mathcal{G}'$ be this isomorphic copy of $\mathcal{G}$. Gluing two copies of $\mathcal{G}$ together along the hexagonal face $F$, identifying vertices $A$
    with $10$, $B$ with $4$, and $C$ with $9$, yields a map $\mathcal{G} \cup_{\partial F} \mathcal{G}$ on an oriented surface of genus 2.

    Moreover, $\mathcal{G} \cup_{\partial F} \mathcal{G}$ is $3$-connected and all of its faces are closed 2-cells that intersect either
    in a vertex, in an edge, or not at all. Therefore, $\mathcal{G} \cup_{\partial F} \mathcal{G}$ is a polyhedral map.  In a similar fashion, we can construct
    $\mathcal{G}' \cup_{\partial F'} \mathcal{G}'$, gluing two copies of $\mathcal{G}'$ together along the boundary of the face $F'$ isomorphic to $F$ via
    $\psi^{-1}$, identifying the vertices as before.  Every face of $\mathcal{G} \cup_{\partial F} \mathcal{G}$ and $\mathcal{G}' \cup_{\partial F'}
    \mathcal{G}'$ is either a 4-gon or a 5-gon. Every vertex of $\mathcal{G} \cup_{\partial F} \mathcal{G}$ and $\mathcal{G}' \cup_{\partial F'} \mathcal{G}'$
    has degree 4 or 5.  In particular, $\mathcal{G} \cup_{\partial F} \mathcal{G}$ and $\mathcal{G}' \cup_{\partial F'} \mathcal{G}'$ are isomorphic via $\psi$ extended to the two copies.

    Notice that $\mathcal{G}' \cup_{\partial F'} \mathcal{G}'$ is dual to $\mathcal{G} \cup_{\partial F} \mathcal{G}$. Indeed, in Figure \ref{fig: square tiling
      torus with 1 hexagon}, we can see which vertices of $\mathcal{G}'$ correspond to which faces of $\mathcal{G}$. This correspondence is carried to
    $\mathcal{G}' \cup_{\partial F'} \mathcal{G}'$.  Thus, $\mathcal{G} \cup_{\partial F} \mathcal{G}$ is a self-dual polyhedral map.

    For genus 3, consider the square tiling of the torus with two triples of squares replaced. This map $\mathcal{H}$ is depicted in
    Figure~\ref{fig: square tiling torus with 2 hexagons}. The map $\mathcal{H'}$ depicted in red in Figure~\ref{fig: square tiling torus with 2 hexagons} is a map isomorphic to
    $\mathcal{H}$, where the isomorphism $\psi$ is given by sending vertices with the same label to each other. Gluing each hexagonal face $F_i$ to a different copy of $\mathcal{G}$, identifying the vertices $A-10, B-4, C-9$ for one face and
    $D-10, E-4, F-9$ for the other one, yields a map $\mathcal{G} \cup_{\partial F_1} \mathcal{H} \cup_{\partial F_2} \mathcal{G}$ on an oriented surface of genus
    3. In a similar fashion as before, $\mathcal{G} \cup_{\partial F_1} \mathcal{H} \cup_{\partial F_2} \mathcal{G}$ is a self-dual polyhedral map.

    For genus $g \ge 4,$ we consider $g-2$ copies of $\mathcal{H}$ and two copies of $\mathcal{G}$. We glue the hexagonal face of a copy of $\mathcal{G}$ to a
    hexagonal face of a copy of $\mathcal{H}$ in the same fashion described before. We then glue the second hexagonal face of $\mathcal{H}$ to the first
    hexagonal face of another copy of $\mathcal{H}$. We continue gluing copies of $\mathcal{H}$ together in this fashion. In the end, we glue the remaining
    hexagonal face of the last copy of $\mathcal{H}$ to the hexagonal face of the second copy of $\mathcal{G}$. This yields a map $\mathcal{G} \cup_{\partial
      F_1} \mathcal{H} \cup_{\partial F_2} \mathcal{H} \cup_{\partial F_3} \mathcal{H} \cup_{\partial F_4} \dots \cup_{\partial F_{g-2}} \mathcal{H}
    \cup_{\partial F_{g-1}} \mathcal{G}$, which is a self-dual polyhedral map on an oriented surface of genus $g$.
\end{proof}

\begin{corollary}\label{cor:ambo}
    There are polyhedral maps of every genus for which ambo increases the symmetry.
\end{corollary}

For each c3-lsp-operation with inflation factor at most $6$ we will now give the set of genera on which the operation can increase the symmetry. The results are given in Table~\ref{tab:small_ops_to5} and Table~\ref{tab:small_ops_6}.
In each row, operations are given that are equivalent in the sense that each can be written as a product with the dual operation of any other. So due to Lemma~\ref{lem:dual_operations} it is sufficient to determine the set of genera
for an arbitrary of the four operations. As ambo can increase the symmetry in every genus and every c3-lsp-operation at least preserves symmetry, an operation that can be written as $O \circ A$ with a c3-lsp-operation $O$ can increase the symmetry in every genus.

{\small \begin{table}[th]
\begin{tabular}{|c|llll|c|}
    \hline
    Inflation & & & & & Can increase \\
    factor & $O$ & $D\circ O$ & $O\circ D$& $D\circ O \circ D$  & symmetry \\
    & & & & & in genus\\
    \hline
    1 & \scalebox{0.6}{\input{decos1.tikz}}& \scalebox{0.6}{\input{decos1_dO.tikz}}& & & $\emptyset$ \\
     & \textbf{Identity} & \textbf{Dual} & & &\\
    \hline
    2 & \scalebox{0.6}{\input{decos2.tikz}}& \scalebox{0.6}{\input{decos2_dO.tikz}}& & &$\mathbb{N}$ \\
    & \textbf{Ambo} & \textbf{Join} & & &\\
    \hline
    3 & \scalebox{0.6}{\input{decos3.tikz}}& \scalebox{0.6}{\input{decos3_dO.tikz}} & \scalebox{0.6}{\input{decos3_Od.tikz}}& \scalebox{0.6}{\input{decos3_dOd.tikz}}& $\mathbb{N}\setminus \{0\}$ \\
    & \textbf{Truncate} & \textbf{Needle} &\textbf{Zip} &\textbf{Kis} &\\
    \hline
    4 & \scalebox{0.6}{\input{decos4a.tikz}} &\scalebox{0.6}{\input{decos4a_dO.tikz}} & & & $\mathbb{N}$ \\
    & \textbf{Expand} & \textbf{Ortho} & & &\\
    \hline
    4 & \scalebox{0.6}{\input{decos4b.tikz}}& \scalebox{0.6}{\input{decos4b_dO.tikz}}& \scalebox{0.6}{\input{decos4b_Od.tikz}} &\scalebox{0.6}{\input{decos4b_dOd.tikz}} & $\{1\}$ \\
    & \textbf{Chamfer} &  & &\textbf{Subdivide} &\\
    \hline
    5 & \scalebox{0.6}{\input{decos5a.tikz}}& \scalebox{0.6}{\input{decos5a_dO.tikz}}& \scalebox{0.6}{\input{decos5a_Od.tikz}} &\scalebox{0.6}{\input{decos5a_dOd.tikz}}& $\emptyset$ \\
    &  &  & &\textbf{Loft} &\\
    \hline
    \end{tabular}
    \caption{All c3-lsp-operations with inflation factor at most 5. The third column gives the set of genera in which the operation can increase symmetry.} \label{tab:small_ops_to5}
\end{table}}

{\small \begin{table}[th]
\begin{tabular}{|c|llll|c|}
    \hline
    Inflation & & & & & Can increase \\
    factor & $O$ & $D\circ O$ & $O\circ D$& $D\circ O \circ D$  & symmetry \\
     & & & & & in genus\\
    \hline
    6 & \scalebox{0.6}{\input{decos6a.tikz}}& \scalebox{0.6}{\input{decos6a_dO.tikz}}& \scalebox{0.6}{\input{decos6a_Od.tikz}} &\scalebox{0.6}{\input{decos6a_dOd.tikz}}& $\emptyset$ \\
    & $O_{6a}$ &  & & &\\
    \hline
    6 & \scalebox{0.6}{\input{decos6b.tikz}} &\scalebox{0.6}{\input{decos6b_dO.tikz}} & & & $\mathbb{N}$ \\
    & $O_{6b}$ &  & & &\\
    \hline
    6 & \scalebox{0.6}{\input{decos6c.tikz}}& \scalebox{0.6}{\input{decos6c_dO.tikz}}& & & $\mathbb{N}$ \\
    & \textbf{Bevel} & \textbf{Meta} & & &\\
    \hline
    6 & \scalebox{0.6}{\input{decos6d.tikz}}& \scalebox{0.6}{\input{decos6d_dO.tikz}}& \scalebox{0.6}{\input{decos6d_Od.tikz}} &\scalebox{0.6}{\input{decos6d_dOd.tikz}}& $\emptyset$ \\
    & $O_{6d}$  &  & &\textbf{Join-lace} &\\
    \hline
    6 & \scalebox{0.6}{\input{decos6e.tikz}}& \scalebox{0.6}{\input{decos6e_dO.tikz}}& \scalebox{0.6}{\input{decos6e_Od.tikz}} &\scalebox{0.6}{\input{decos6e_dOd.tikz}}& $\mathbb{N}\setminus \{0\}$ \\
    & $O_{6e}$ &  & & &\\
    \hline
    6 & \scalebox{0.6}{\input{decos6f.tikz}}& \scalebox{0.6}{\input{decos6f_dO.tikz}}& \scalebox{0.6}{\input{decos6f_Od.tikz}} &\scalebox{0.6}{\input{decos6f_dOd.tikz}}& $\emptyset$ \\
    & \textbf{Quinto} &  & & &\\
    \hline
   \end{tabular}
    \caption{All c3-lsp-operations with inflation factor 6. The third column gives the set of genera in which the operation can increase symmetry.} \label{tab:small_ops_6}
\end{table}
}

    \begin{itemize}
        \item[1:] The identity operation can obviously not increase symmetry.
        \item[2:] The result for ambo is proven in Corollary~\ref{cor:ambo}.
        \item[3:] For truncation the result is exactly Corollary~\ref{cor:bitruncation}.
        \item[4a:] The operation {\em expand} can be written as $A \circ A$ and can therefore increase the symmetry in every genus. 
        \item[4b:] The operation {\em chamfer} is the operation $\mathrm{GC}(2,0)$, so it follows by Lemma~\ref{lem:GC_colour0_allgenus} and Lemma~\ref{lem:GC_genus1} that chamfer can increase symmetry only on genus $1$.
        \item[5:] We will prove the result for the operation {\em loft}. Let $P$ be a polyhedral map and $L(P)$ the polyhedral map obtained by applying loft to $P$. In $L(P)$ the vertices labeled $v_0$ double their degree from $P$, so
    they have degree at least 6, as $P$ is a polyhedral map. On the other hand, the new vertices introduced by the loft operation have degree 3. Therefore, any
    automorphism of $L(P)$ maps vertices of class $v_0$ to vertices of the same class.

    Vertices of class $v_0$ belong to two classes of chambers. One of them contains (half) an edge of $P$ that leads to a vertex of degree $3$ and one of them contains
    (half) an edge that leads to a vertex of degree at least $6$. So an automorphism can never map a chamber containing a vertex of class $v_0$ to a chamber in a different class
    and the result follows from Lemma~\ref{lem:map_chamber_to_same_class}.

  \item[6a:] Let $P$ be a polyhedral map. The vertices in $O_{6a}(P)$ have two different degrees -- once $3$ and once $6$.
    Vertices of class $v_2$ are the only colour $2$  vertices in the barycentric subdivision of $O_{6a}(P)$ that neighbour only colour $0$ vertices with degree $6$ in $O_{6a}(P)$. So any automorphism must map
    vertices of class $v_2$ onto vertices of class $v_2$ and as these vertices belong to only one class of chambers, the result follows with  Lemma~\ref{lem:map_chamber_to_same_class}.

        \item[6b:] $O_{6b}$ is $T \circ D \circ A$. It follows that $O_{6b}$ can increase symmetry in any genus. 

        \item[6c:] The operation {\em bevel} is $T \circ A$. It follows that bevel can increase symmetry in any genus.

        \item[6d:]  Let $P$ be a polyhedral map. The vertices in $O_{6d}(P)$ of class $v_0$ are the only vertices that are only contained in $4$-gons, as the faces corresponding to the vertex of class $v_2$ are at least hexagons.
          So vertices of class $v_0$ must be mapped on other vertices of class $v_0$
          by any automorphism. As they are only contained in chambers of the same class, the result follows from Lemma~\ref{lem:map_chamber_to_same_class}.

        \item[6e:] Operation $O_{6e}$ can be written as $A \circ T$, so it follows from Lemma~\ref{lem:truncationfrom1} that it can increase the symmetry in every genus higher than $0$.
           We still have to consider genus $0$. Let $P$ be a
           polyhedron. Truncation cannot increase the symmetry in the plane, but ambo can.  As truncation cannot increase the symmetry of a polyhedron, ambo must increase the symmetry of $T(P)$ if $A \circ T(P)$ has more symmetry than $P$,
           so $T(P)$ must be self-dual. As the result of truncation is -- no matter on which genus -- is always a 3-regular map that also contains faces of size at least $6$ (as already mentioned in the proof of
           Lemma~\ref{lem: truncation does not increase symmetry on plane}), the result of truncation is never self-dual and $O_{6e}= A \circ T$ never increases symmetry on genus $0$. Note that this argument implies that $O_{6e}$ increases symmetry for
           exactly the same polyhedral maps as truncation and by the same factor.
           
         \item[6f:] Let $Q(P)$ be the result of applying {\em quinto} to a polyhedral map. Then the vertices of class $v_0$ are the only vertices that neighbour only vertices of degree $4$. So the vertices of class $v_0$
           are mapped onto vertices of class $v_0$ by any automorphism of $Q(P)$ and as these vertices are only contained in chambers of the same class, it follows from  Lemma~\ref{lem:map_chamber_to_same_class}
           that quinto cannot increase the symmetry.
	\end{itemize}

\section{Future work}

The most captivating question is whether each c3-lsp-operation that can increase the symmetry of a polyhedron can be written as a product of another operation with ambo, so that ambo is essentially the only operation that can do it.
We have proven that
this is the case for all c3-lsp-operations with inflation factor up to $6$, but though the result could be extended to slightly larger inflation factors, it is still not known whether it is true in general.
Solving this question would again emphasize the special role ambo plays among all c3-lsp-operations as well as the special role of the plane among all surfaces.

When writing {\em a product of another operation with ambo} it says nothing about the order of operations. In fact so far we have only examples of operations that can increase the symmetry of polyhedra where
the product is of the form $O\circ A$ -- with $O$ an arbitrary operation. It is not clear whether operations can make self-dual polyhedra out of polyhedra that are not self-dual.
If that is true, also operations of the form $A\circ O$ with $O$ not being of the form $O'\circ A$ could increase the symmetry of polyhedra.

 In this text the arguments used for different operations differ from each other.
 In \cite{camp2023effect} a simple criterion is given to judge whether a given operation preserves 3-connectivity of maps.
 Such a general criterion that makes it easy to judge whether a given operation can increase symmetry (or the opposite:
cannot increase symmetry on certain genera) would be a useful achievement. For 3-connectivity it does not make a difference whether one studies the map or the
underlying graph, for other invariants it does. Some of the results proven in this article imply corresponding results for the automorphism group of the
underlying abstract graph of the map. E.g.\ for polyhedra the size of the automorphism group of the map is the same as for the abstract graph (a consequence of Whitney's unique embedding theorem \cite{Whitney_unique_embed}) --
for higher genus this is not necessarily the case and some of the results might not hold. Given a map $P$ and a c3-lsp-operation $O$, the genus of $O(P)$ is
obviously the same as that of $P$. But even if $P$ is a minimum genus embedding of the underlying abstract graph, so that the genus of the graph (defined as the
minimum genus in which it can be embedded) and that of the map coincide, the genus of the underlying graph of $O(P)$ is always the same for some operations and
can differ a lot for others. There is e.g. for each genus $g$ a map that is a minimum genus embedding of the underlying graph, but where the underlying graph of the dual is planar.

So there are many open problems for c3-lsp-operations and of course all these problems must also be posed for the more general class of c3-lopsp-operations that can destroy symmetries.
At least one of the problems is easier for the more general class: for c3-lopsp-operations it is well known
that they can increase the symmetry of polyhedra even if they are not a product with ambo. E.g.\ the c3-lopsp-operation {\em snub} applied to a tetrahedron produces an icosahedron.

\bibliographystyle{plain}
\bibliography{bibliography}

\end{document}